\documentclass [11pt,twoside,a4paper]{article}
\usepackage{amsfonts}
\usepackage{amsthm}
\usepackage{amsmath}
\usepackage{amstext}
\usepackage{amssymb}
\usepackage{mathrsfs}
\usepackage{amscd}
\usepackage{xypic}
\usepackage{epsf}              
\usepackage{graphicx}          
\usepackage{fancybox}          
\usepackage{color}             
\usepackage{fancyhdr}
\usepackage[hang,footnotesize]{caption2}  

\def\mathcal{\mathscr}
\newfont{\aaa}{cmb10 at 19pt}
\newfont{\bbb}{cmb10 at 11pt}
\newtheorem{lem}{Lemma}[]
\newtheorem{thm}{Theorem}

\newtheorem{rem}{Remark}[]
\newtheorem{definition}{Definition}[]
\newtheorem{conjecture}{Conjecture}[]
\newtheorem{pro}{Proposition}[]
\def\ZZ{\mathbb Z}

\newcommand{\beq}{\begin{equation}}
\newcommand{\eeq}{\end{equation}}
\newcommand{\bey}{\begin{eqnarray}}
\newcommand{\eey}{\end{eqnarray}}
\newcommand{\beyy}{\begin{eqnarray*}}
\newcommand{\eeyy}{\end{eqnarray*}}

 \makeatletter
\def\@evenhead{
   \vbox{\hbox to \textwidth
{}{\hspace{0mm}{\footnotesize \thepage}}{\hspace{8cm}
   {\footnotesize {Ming Ding}}}
  \protect\vspace{1truemm}\relax
   \hrule depth0pt height0.15truemm width\textwidth
   }}
   \def\@evenfoot{}
\def\@oddhead{
    \vbox{\hbox to \textwidth
  {{\hspace{0cm}{\footnotesize On  quantum cluster algebras of finite
type}
  \hfill{\footnotesize \thepage}}\hspace{0mm}}{}
   \protect\vspace{1truemm}\relax
   \hrule depth0pt height0.15truemm width\textwidth
  }}
  \def\@oddfoot{}
\makeatother

\begin{document}



\setcounter{page}{1}
\qquad\\[8mm]

\noindent{\aaa{On  quantum cluster algebras of finite
type}}\\[1mm]

\noindent{\bbb Ming Ding}\\[-1mm]

\noindent\footnotesize{Institute for advanced study, Tsinghua University, Beijing 100084, China}\\[6mm]


\normalsize\noindent{\bbb Abstract}\quad We extend the definition of
a quantum analogue of the Caldero-Chapoton map defined \cite{rupel}.
When $Q$ is a quiver of finite type, we prove that the algebra
$\mathcal{AH}_{|k|}(Q)$ generated by all cluster characters (see
Definition \ref{def}) is exactly the quantum cluster algebra
$\mathcal{EH}_{|k|}(Q)$.\vspace{0.3cm}

\footnotetext{E-mail: m-ding04@mails.tsinghua.edu.cn}

\noindent{\bbb Keywords}\quad cluster variable, quantum cluster algebra\\
{\bbb MSC}\quad 16G20\\[0.4cm]

\noindent{\bbb{1\quad Introduction}}\\[0.1cm]
Quantum cluster algebras  were introduced by A. Berenstein and A.
Zelevinsky \cite{berzel}  to  study  the canonical basis.  When
$q=1,$ the quantum cluster algebras are exactly the corresponding
cluster algebras which were introduced and studied  by S. Fomin and
A. Zelevinsky in a series of papers \cite{ca1}\cite{ca2}\cite{BFZ}.
A quantum analogue of the Caldero-Chapoton formula \cite{caldchap}
was defined by D.~Rupel \cite{rupel} and the author
 conjectured that cluster variables could be expressed using this formula and proved it for the cluster variables in finite types as well
as in almost acyclic clusters. Later this conjecture was confirmed
for acyclic equally valued quivers in \cite{fanqin}. Quantum cluster
algebra  structures  have been studied in a few cases, see for
example
\cite{j}\cite{rupel}\cite{lampe}\cite{DX0}\cite{fanqin}\cite{DX}.

The cluster category was  introduced for its combinatorial
similarities with cluster algebras. In contrast to the case of
cluster algebras, for any objects $M,N$ in the cluster category
associated to a quantum cluster algebra, it does not generally hold
that $X_{N}X_{M}=|k|^{\pm\frac{1}{2}n_{N\oplus M}}X_{N\oplus M}$ for
any $n_{N\oplus M}\in \ZZ$. Thus the natural problem  is to ask if
$X_{N\oplus M}$ is in the corresponding quantum cluster algebra.
Hence it becomes interesting to study the relation between the
algebra generated by all cluster characters (see Definition
\ref{def}) and the corresponding quantum cluster algebra. In the
case of cluster algebras, these are equal for finite and affine
types \cite{CK2005}\cite{DXX}. In \cite{g1}\cite{g2}, C. Geiss, B.
Leclerc and J. Schr$\ddot{o}$er have proved that a large class of
cluster algebras always contain cluster characters of all objects in
the cluster categories. The aim of this article is to prove that for
any quiver $Q$ of finite type, the algebra $\mathcal{AH}_{|k|}(Q)$
generated by all cluster characters is still the quantum cluster
algebra $\mathcal{EH}_{|k|}(Q)$.
\noindent \\[4mm]

\noindent{\bbb 2\quad Preliminaries and statement of the main result}\\[0.1cm]
\noindent{\bbb 2.1\quad Definition of quantum cluster algebras}
 Let $L$ be a lattice of rank $m$ and $\Lambda:L\times L\to
\ZZ$ a skew-symmetric bilinear form. Note that $\Lambda$ can  be
identified with an $m\times m$ skew-symmetric matrix which still
denoted by $\Lambda$ if there is no confusion. Set a formal variable
$q$ and
 the ring of integer Laurent polynomials $\ZZ[q^{\pm1/2}]$.
Define the \textit{based quantum torus} associated to the pair
$(L,\Lambda)$ to be the $\ZZ[q^{\pm1/2}]$-algebra $\mathcal{T}$ with
a distinguished $\ZZ[q^{\pm1/2}]$-basis $\{X^e: e\in L\}$ and the
multiplication
\[X^eX^f=q^{\Lambda(e,f)/2}X^{e+f}.\]
  It is known that $\mathcal{T}$ is  contained in its
skew-field of fractions $\mathcal{F}$. A \textit{toric frame} in
$\mathcal{F}$ is a map $M: \ZZ^m\to \mathcal{F} \setminus \{0\}$
given by
\[M({\bf c})=\varphi(X^{\eta({\bf c})})\] where $\varphi$ is an
automorphism of $\mathcal{F}$ and $\eta: \ZZ^m\to L$ is an
isomorphism of lattices. By the definition, the elements $M({\bf
c})$ form a $\ZZ[q^{\pm1/2}]$-basis of the based quantum torus
$\mathcal{T}_M:=\varphi(\mathcal{T})$ and satisfy the following
relations:
\[M({\bf c})M({\bf d})=q^{\Lambda_M({\bf c},{\bf d})/2}M({\bf c}+{\bf d}),\
M({\bf c})M({\bf d})=q^{\Lambda_M({\bf c},{\bf d})}M({\bf d})M({\bf
c}),\]
\[ M({\bf 0})=1,\ M({\bf c})^{-1}=M(-{\bf c}),\]
where $\Lambda_M$ is the skew-symmetric bilinear form on $\ZZ^m$
obtained from the lattice isomorphism $\eta$.  Let $\Lambda_M$ be
 the skew-symmetric $m\times m$ matrix defined by
$\lambda_{ij}=\Lambda_M(e_i,e_j)$ where $\{e_1, \ldots, e_m\}$ is
the standard basis of $\ZZ^m$.  Given a toric frame $M$, let
$X_i=M(e_i)$.  Then we have
$$\mathcal{T}_M=\ZZ[q^{\pm1/2}]\langle X_1^{\pm 1}, \ldots,
X_m^{\pm1}:X_iX_j=q^{\lambda_{ij}}X_jX_i\rangle.$$  An easy
computation shows that:
\[M({\bf c})=q^{\frac{1}{2}\sum_{i<j}
c_ic_j\lambda_{ji}}X_1^{c_1}X_2^{c_2}\cdots X_m^{c_m}=:X^{({\bf c})}
\ \ \ ({\bf c}\in\ZZ^m).\]

Let $\Lambda$ be an $m\times m$ skew-symmetric matrix and
$\tilde{B}$  an $m\times n$ matrix with $n\le m$.  We call the pair
$(\Lambda, \tilde{B})$ \textit{compatible} if up to permuting rows
and columns $\tilde{B}^T\Lambda=(D|0)$  with
$D=diag(d_1,\cdots,d_n)$ where $d_i\in \mathbb{N}$ for $1\leq i\leq
n$. The pair $(M,\tilde{B})$ is called a \textit{quantum seed} if
the pair $(\Lambda_M, \tilde{B})$ is compatible.  Define the
$m\times m$ matrix $E=(e_{ij})$ as follows
\[e_{ij}=\begin{cases}
\delta_{ij} & \text{if $j\ne k$;}\\
-1 & \text{if $i=j=k$;}\\
max(0,-b_{ik}) & \text{if $i\ne j = k$.}
\end{cases}
\]
For $n,k\in\ZZ$, $k\ge0$, denote ${n\brack
k}_q=\frac{(q^n-q^{-n})\cdots(q^{n-k+1}-q^{-n+k-1})}{(q^k-q^{-k})\cdots(q-q^{-1})}$.
Let $k\in[1,n]$ where $[1,n]=\{1,\cdots,n\}$ and ${\bf
c}=(c_1,\ldots,c_m)\in\ZZ^m$ with $c_{k}\geq 0$. Define the toric
frame $M': \ZZ^m\to \mathcal{F} \setminus \{0\}$ as follows
\begin{equation}\label{eq:cl_exp}M'({\bf c})=\sum^{c_k}_{p=0} {c_k \brack p}_{q^{d_k/2}} M(E{\bf c}+p{\bf b}^k),\ \ M'({\bf -c})=M'({\bf c})^{-1}.\end{equation}
where the vector ${\bf b}^k\in\ZZ^m$ is the $k$th column of
$\tilde{B}$.  Following \cite{ca1}, we say a real $m\times n$ matrix
$\tilde{B}'$ is obtained from $\tilde{B}$  by matrix mutation in
direction $k$ if the entries of $\tilde{B}'$ are given by
\[b'_{ij}=\begin{cases}
-b_{ij} & \text{if $i=k$ or $j=k$;}\\
b_{ij}+\frac{|b_{ik}|b_{kj}+b_{ik}|b_{kj}|}{2} & \text{otherwise.}
\end{cases}
\]
Then the quantum seed $(M',\tilde{B}')$ is defined to be the
mutation of $(M,\tilde{B})$ in direction $k$. Two quantum seeds are
called mutation-equivalent if they can be obtained from each other
by a sequence of mutations. Let $\mathcal{C}=\{M'(e_i): i\in[1,n]\}$
where $(M',\tilde{B}')$ is mutation-equivalent to $(M,\tilde{B})$.
The elements of $\mathcal{C}$ are called the \textit{cluster
variables}. Let $\mathbb{P}=\{M(e_i): i\in[n+1,m]\}$ and  the
elements of $\mathbb{P}$ are called \emph{coefficients}. Denote by
$\ZZ\mathbb{P}$ the ring of Laurent polynomials generated by
$q^{\frac{1}{2}},\mathbb{P}$ and their inverses. Then the
\textit{quantum cluster algebra}
$\mathcal{A}_q(\Lambda_M,\tilde{B})$ is defined to be the
$\ZZ\mathbb{P}$-subalgebra of $\mathcal{F}$ generated by
$\mathcal{C}$.\\[0.1cm]

\noindent{\bbb 2.2\quad The quantum Caldero-Chapoton map and  main
result} Let $k$ be a finite field with cardinality $|k|=q$ and
$m\geq n$ be two positive integers and $\widetilde{Q}$ an acyclic
valued quiver with vertex set $\{1,\ldots,m\}$. Denote the subset
$\{n+1,\dots,m\}$ by $C$. The full subquiver $Q$ on the vertices
$1,\ldots,n$ is called the \emph{principal part} of $\widetilde{Q}$.
For $1\leq i\leq m$, let $S_i$ be the $i$th simple module for
$k\widetilde{Q}.$

Let $\widetilde{B}$ be the $m\times n$ matrix associated to the
quiver $\widetilde{Q}$ whose entry in position $(i,j)$ given by
\[
b_{ij}=|\{\mathrm{arrows}\, i\longrightarrow
j\}|-|\{\mathrm{arrows}\, j\longrightarrow i\}|
\]
for $1\leq i\leq m$, $1\leq j\leq n$. Denote by $\widetilde{I}$ the
left $m\times n$ submatrix of the identity matrix of size $m\times
m$. Assume that there exists some antisymmetric $m\times m$ integer
matrix $\Lambda$ such that
\begin{align}\label{eq:simply_laced_compatible}
\Lambda(-\widetilde{B})=\begin{bmatrix}I_n\\0
\end{bmatrix},
\end{align}
where $I_n$ is the identity matrix of size $n\times n$. Let
$\widetilde{R}=\widetilde{R}_{\widetilde{Q}}$ be the $m\times n$
matrix with its entry in position $(i,j)$ given by
\[
\widetilde{r}_{ij}:=\mathrm{dim}_{k}\mathrm{Ext}^{1}_{k\widetilde{Q}}(S_j,S_i)=|\{\mathrm{arrows}\,
j\longrightarrow i\}|.
\] for $1\leq i\leq m$, $1\leq j\leq n$. Set
$\widetilde{R}^{tr}=\widetilde{R}_{\widetilde{Q}^{op}}.$
 Denote the principal $n\times n$ submatrices of
 $\widetilde{B}$ and $\widetilde{R}$ by $B$ and $R$
respectively. Note that
$\widetilde{B}=\widetilde{R}^{tr}-\widetilde{R}$ and $B=R^{tr}-R$.

Let $\mathcal C_{\widetilde{Q}}$ be the cluster category of $k
\widetilde{Q}$, i.e., the orbit category of the derived category
$\mathcal{D}^b(\widetilde{Q})$ under the action of  the functor
$F=\tau\circ[-1]$ (see \cite{BMRRT}). Let  $I_i$ be the
indecomposable injective $k \widetilde{Q}$ module for $1\leq i \leq
m.$ Then the indecomposable $k \widetilde{Q}$-modules and $I_i[-1]$
for $1\leq i \leq m$ exhaust all indecomposable objects of the
cluster category $\mathcal C_{\widetilde{Q}}$. Each object $M$ in
$\mathcal C_{\widetilde{Q}}$ can be uniquely decomposed as
$$M=M_0\oplus I_M[-1]$$
where $M_0$ is a module and $I_M$ is an injective  module.

The Euler form on $k \widetilde{Q}$-modules $M$ and $N$ is given by
$$\langle M,N\rangle=\mathrm{dim}_{k}\mathrm{Hom}(M,N)-\mathrm{dim}_{k}\mathrm{Ext}^{1}(M,N).$$
Note that the Euler form only depends on the dimension vectors of
$M$ and $N$.

The quantum Caldero-Chapoton map of an acyclic quiver
$\widetilde{Q}$ has been defined in \cite{rupel} and \cite{fanqin}.
In \cite{rupel}, the author defined the quantum Caldero-Chapoton map
for $k\widetilde{Q}$-modules while in \cite{fanqin} for
coefficient-free rigid object in $\mathcal C_{\widetilde{Q}}$. For
our purpose, we need to extend these definitions to the following
map
$$X_?: \mathrm{obj}\mathcal C_{\widetilde{Q}}\longrightarrow \mathcal{T}$$
defined by the following rules:\\
(1)\ If $M$ is a $k Q$-module, then
                    $$
                       X_{M}=\sum_{\underline{e}} |\mathrm{Gr}_{\underline{e}} M|q^{-\frac{1}{2}
\langle
\underline{e},\underline{m}-\underline{e}-\underline{i}\rangle}X^{-\widetilde{B}\underline{e}-(\widetilde{I}-\widetilde{R}^{tr})\underline{m}};$$
(2)\ If $M$ is a $k Q$-module and $I$ is an injective $k
\widetilde{Q}$-module, then
                    $$
                       X_{M\oplus I[-1]}=\sum_{\underline{e}} |\mathrm{Gr}_{\underline{e}} M|q^{-\frac{1}{2}
\langle
\underline{e},\underline{m}-\underline{e}-\underline{i}\rangle}X^{-\widetilde{B}\underline{e}-(\widetilde{I}-\widetilde{R}^{tr})\underline{m}+\underline{\mathrm{dim}}
soc I},
                    $$
where $\underline{\mathrm{dim}} I= \underline{i},
\underline{\mathrm{dim}} M= \underline{m}$ and
$\mathrm{Gr}_{\underline{e}}M$ denotes the set of all submodules $V$
of $M$ with $\underline{\mathrm{dim}} V= \underline{e}$. We note
that
$$
X_{P[1]}=X_{\tau P}=X^{\underline{\mathrm{dim}} P/rad
P}=X^{\underline{\mathrm{dim}}\mathrm{soc}I}=X_{I[-1]}=X_{\tau^{-1}I}.
$$
for any projective $k\widetilde{Q}$-module $P$ and injective
$k\widetilde{Q}$-module $I$ with $\mathrm{soc}I=P/\mathrm{rad}P.$ In
the following, we denote by the corresponding underlined lower case
 letter $\underline{x}$ the dimension vector of a $kQ$-module
$X$ and view $\underline{x}$ as a column vector in $\mathbb{Z}^n.$

\begin{definition}\label{def}
$X_{L}$ is called \emph{the corresponding cluster character}, if $L$
is a $k Q$-module or $L=M\oplus I[-1]\in\mathcal C_{\widetilde{Q}}$
satisfying that $M$ is a $k Q$-module and $I$ is an injective $k
\widetilde{Q}$-module.
\end{definition}
For a quiver $Q$, denote by $\mathcal{AH}_{|k|}(Q)$ the
 $\mathbb{ZP}$-subalgebra of $\mathcal{F}$ generated by
all the cluster characters and by $\mathcal{EH}_{|k|}(Q)$ the
corresponding quantum cluster algebra, i.e, the
$\mathbb{ZP}$-subalgebra of $\mathcal{F}$ generated by all the
cluster variables. Note that here we are working over a finite
field, the definition of quantum cluster algebra in section 2.1
remains valid (see \cite{fanqin}). The main result of this article
is the following theorem:

\begin{thm}\label{main}
For any  quiver $Q$ of finite type, we have
$\mathcal{EH}_{|k|}(Q)=\mathcal{AH}_{|k|}(Q).$
\end{thm}
We conjecture that Theorem \ref{main} holds for any quiver of affine
type.
\begin{conjecture}
For any  quiver $Q$ of affine type, we have
$\mathcal{EH}_{|k|}(Q)=\mathcal{AH}_{|k|}(Q).$
\end{conjecture}
\noindent\\[4mm]

\noindent{\bbb 3\quad Proof of the main theorem}\\[0.1cm]
In this section, we fix a  quiver $Q$ of finite type with $n$
vertices. Firstly, we recall some notations. For any
$k\widetilde{Q}-$modules $M,N$ and $E$, denote by
$\varepsilon_{MN}^{E}$  the cardinality of the set
$\mathrm{Ext}_{k\widetilde{Q}}^{1}(M,N)_{E}$ which is the subset of
$ \mathrm{Ext}_{k\widetilde{Q}}^{1}(M,N)$ consisting of those
equivalence classes of short exact sequences with middle term
isomorphic to $E$ (\cite[Section 4]{Hubery}). Let $F^M_{AB}$ be the
number of submodules $U$ of $M$ such that $U$ is isomorphic to $B$
and $M/U$ is isomorphic to $A$. Then by definition, we have
$$|\mathrm{Gr}_{\underline{e}}(M)|=\sum_{A, B;
\underline{\mathrm{dim}}B=\underline{e}}F_{AB}^M.
$$
Denote by
$[M,N]^{1}=\mathrm{dim}_{k}\mathrm{Ext}_{k\widetilde{Q}}^{1}(M,N)$
and $[M,N]=\mathrm{dim}_{k}\mathrm{Hom}_{k\widetilde{Q}}(M,N).$
 The following Theorem \ref{hall multi} proved in \cite{DX} and Proposition
\ref{exchange2} give the explicit relations between $X_{N}X_{M}$ and
$X_{N\oplus M}$.
\begin{thm}{\cite{DX}}\label{hall multi}
Let $M$ and $N$ be $kQ$-modules. Then
$$q^{[M,N]^{1}}X_{M}X_{N}=q^{\frac{1}{2}\Lambda((\widetilde{I}-\widetilde{R}^{tr})\underline{m},
(\widetilde{I}-\widetilde{R}^{tr})\underline{n})}
\sum_{E}\varepsilon_{MN}^{E}X_E.$$
\end{thm}

Let $M$ be any $kQ-$module  and $I$ any injective
$k\widetilde{Q}-$module. Define
$$\mathrm{Hom}_{k\widetilde{Q}}(M,I)_{BI'}:=\{f:M\longrightarrow I|ker f\cong B,coker f\cong
I'\}.$$ Note that $I'$ is an injective $k\widetilde{Q}-$module. The
following result, together with Theorem \ref{hall multi}, is
essential for us to prove Theorem \ref{main}.
\begin{pro}\label{exchange2}
With the above notations, we have
$$q^{[M,I]}X_{M}X_{I[-1]}=q^{\frac{1}{2}\Lambda((\widetilde{I}-\widetilde{R}^{tr})\underline{m},
-\mathrm{\underline{dim}}soc I)}
\sum_{B,I'}|\mathrm{Hom}_{k\widetilde{Q}}(M,I)_{BI'}|X_{B\oplus
I'[-1]}.$$
\end{pro}
\begin{proof}We calculate
\begin{eqnarray}
   && X_{M}X_{I[-1]}  \nonumber\\
   &=& \sum_{G,H}q^{-\frac{1}{2}\langle H,G\rangle}F^{M}_{GH}X^{-\widetilde{B}\underline{h}-(\widetilde{I}-\widetilde{R}^{tr})\underline{m}}
   X^{\mathrm{\underline{dim}}soc I}\nonumber\\
  &=& \sum_{G,H}q^{-\frac{1}{2}\langle H,G\rangle}F^{M}_{GH}q^{\frac{1}{2}\Lambda(-\widetilde{B}\underline{h}-(\widetilde{I}-\widetilde{R}^{tr})\underline{m},
\mathrm{\underline{dim}}soc
I)}X^{-\widetilde{B}\underline{h}-(\widetilde{I}-\widetilde{R}^{tr})\underline{m}+\mathrm{\underline{dim}}soc
I}
\nonumber\\
 &=&q^{\frac{1}{2}\Lambda(-(\widetilde{I}-\widetilde{R}^{tr})\underline{m},
\mathrm{\underline{dim}}soc I)}\sum_{G,H}q^{-\frac{1}{2}\langle
H,G\rangle}q^{\frac{1}{2}\Lambda(-\widetilde{B}\underline{h},\mathrm{\underline{dim}}soc
I)}F^{M}_{GH}X^{-\widetilde{B}\underline{h}-(\widetilde{I}-\widetilde{R}^{tr})\underline{m}+\mathrm{\underline{dim}}soc
I}\nonumber\\
 &=&q^{\frac{1}{2}\Lambda((\widetilde{I}-\widetilde{R}^{tr})\underline{m},
-\mathrm{\underline{dim}}soc I)}\sum_{G,H}q^{-\frac{1}{2}\langle
H,G\rangle}q^{-\frac{1}{2}[H,I]}F^{M}_{GH}X^{-\widetilde{B}\underline{h}-(\widetilde{I}-\widetilde{R}^{tr})\underline{m}+\mathrm{\underline{dim}}soc
I}.\nonumber
\end{eqnarray}
Here we use the fact that
$$\Lambda(-\widetilde{B}\underline{h},\mathrm{\underline{dim}}soc
I)=-\underline{h}^{tr}\widetilde{B}^{tr}\Lambda(\mathrm{\underline{dim}}soc
I)=-\underline{h}^{tr}(\mathrm{\underline{dim}}soc I)=-[H,I].$$
 Note that if we have the following commutative diagram
$$
\xymatrix{&0\ar[d]&0\ar[d]\\
&Y\ar@{=}[r]\ar[d]&Y\ar[d]\\
 0\ar[r]&B\ar[r]\ar[d]&M\ar[r]\ar[d]&I\ar[r]&I'\ar[r]&0\\
0\ar[r]&X\ar[r]\ar[d]&G\ar[d]\\
&0&0}
$$
and short exact
 sequences
$$0\longrightarrow B\longrightarrow M \longrightarrow A\longrightarrow 0$$ $$0\longrightarrow A\longrightarrow I\longrightarrow I'\longrightarrow 0,$$
then by \cite{Hubery} it follows that
$$\sum_{B}F^{B}_{XY}F^{M}_{AB}=\sum_{G}F^{G}_{AX}F^{M}_{GY},\
|\mathrm{Hom}_{k\widetilde{Q}}(M,I)_{BI'}|=\sum_{A}|\mathrm{Aut}(A)|F^{M}_{AB}F^{I}_{I'A}$$
and
$$\sum_{A,I',X}|\mathrm{Aut}(A)|F^{I}_{I'A}F^{G}_{AX}=\sum_{I',X}|\mathrm{Hom}_{k\widetilde{Q}}(G,I)_{XI'}|=q^{[
G,I]}=q^{\langle G,I\rangle}.$$ By \cite[Lemma 1]{Hubery}, we have
$(\widetilde{I}-\widetilde{R}^{tr})\underline{i}=\mathrm{\underline{dim}}
soc I$. Now we can  calculate the  term
\begin{eqnarray}
   &&  \sum_{B,I'}|\mathrm{Hom}_{k\widetilde{Q}}(M,I)_{BI'}|X_{B\oplus
I'[-1]}  \nonumber\\
   &=& \sum_{A,B,I',X,Y}|\mathrm{Aut}(A)|F^{M}_{AB}F^{I}_{I'A}q^{-\frac{1}{2}\langle
Y,X-I'\rangle}F^{B}_{XY}X^{-\widetilde{B}\underline{y}-(\widetilde{I}-\widetilde{R}^{tr})\underline{b}+\mathrm{\underline{dim}}
soc I'}\nonumber\\
  &=& \sum_{A,G,I',X,Y}q^{-\frac{1}{2}\langle
Y,X-I'\rangle}|\mathrm{Aut}(A)|F^{I}_{I'A}F^{G}_{AX}F^{M}_{GY}X^{-\widetilde{B}\underline{y}-(\widetilde{I}-\widetilde{R}^{tr})\underline{b}+\mathrm{\underline{dim}}
soc I'}.\nonumber
\end{eqnarray}
Note that  we have  the  following facts
$$\underline{i'}+\underline{a}=\underline{i},\
\underline{x}+\underline{a}=\underline{g}\Longrightarrow
\underline{x}-\underline{i'}=\underline{g}-\underline{i},$$ and
\begin{eqnarray}
   &&  -\widetilde{B}\underline{y}-(\widetilde{I}-\widetilde{R}^{tr})\underline{b}+\mathrm{\underline{dim}}
soc I' \nonumber\\
   &=&-\widetilde{B}\underline{h}-(\widetilde{I}-\widetilde{R}^{tr})(\underline{m}-\underline{i}-\underline{i'})+\mathrm{\underline{dim}}
soc I'\nonumber\\
  &=& -\widetilde{B}\underline{h}-(\widetilde{I}-\widetilde{R}^{tr})\underline{m}
  +(\widetilde{I}-\widetilde{R}^{tr})(\underline{i}-\underline{i'})+\mathrm{\underline{dim}}
soc I'
\nonumber\\
 &=& -\widetilde{B}\underline{h}-(\widetilde{I}-\widetilde{R}^{tr})\underline{m}
  +(\widetilde{I}-\widetilde{R}^{tr})\underline{i}
\nonumber\\
 &=&-\widetilde{B}\underline{h}-(\widetilde{I}-\widetilde{R}^{tr})\underline{m}
  +\mathrm{\underline{dim}}
soc I.\nonumber
\end{eqnarray}
Hence
\begin{eqnarray}
   &&  \sum_{B,I'}|\mathrm{Hom}_{k\widetilde{Q}}(M,I)_{BI'}|X_{B\oplus
I'[-1]}  \nonumber\\
 &=&\sum_{G,H}q^{\langle
G,I\rangle}q^{-\frac{1}{2}\langle
H,G-I\rangle}F^{M}_{GH}X^{-\widetilde{B}\underline{h}-(\widetilde{I}-\widetilde{R}^{tr})\underline{m}+\mathrm{\underline{dim}}
soc I}
\nonumber\\
 &=&\sum_{G,H}q^{\langle
M,I\rangle}q^{-\frac{1}{2}\langle H,I\rangle}q^{-\frac{1}{2}\langle
H,G\rangle}F^{M}_{GH}X^{-\widetilde{B}\underline{h}-(\widetilde{I}-\widetilde{R}^{tr})\underline{m}+\mathrm{\underline{dim}}
soc I}\nonumber\\
 &=&q^{[
M,I]}\sum_{G,H}q^{-\frac{1}{2}[H,I]}q^{-\frac{1}{2}\langle
H,G\rangle}F^{M}_{GH}X^{-\widetilde{B}\underline{h}-(\widetilde{I}-\widetilde{R}^{tr})\underline{m}+\mathrm{\underline{dim}}
soc I}.\nonumber
\end{eqnarray}
This finishes the proof.
\end{proof}
\begin{rem}
Proposition \ref{exchange2} holds for any acyclic quiver.
\end{rem}
 The following lemma is well-known. Here we give a sketch of the proof following \cite[Lemma 8(b)]{CK2005}.
\begin{lem}\label{easy}
        Let $$M \longrightarrow E \longrightarrow N \xrightarrow{\epsilon} M[1]$$
        be a non-split triangle in $\mathcal C_{\widetilde{Q}}.$ Then
        $$\mathrm{dim}_{k}\mathrm{Ext}^{1}_{\mathcal C_{\widetilde{Q}}}(E,E) < \mathrm{dim}_{k}\mathrm{Ext}^{1}_{\mathcal C_{\widetilde{Q}}}(M \oplus N, M \oplus N).$$
\end{lem}
\begin{proof}
       For any object $L\in \mathcal C_{\widetilde{Q}}$, applying the  functor $\mathrm{Ext}^{1}_{\mathcal C_{\widetilde{Q}}}(-,L)$  to the above non-split triangle
       gives rise to the exact sequence
        $$0 \longrightarrow ker f_{L} \longrightarrow \mathrm{Ext}^{1}_{\mathcal C_{\widetilde{Q}}}(N,L)\xrightarrow{f_{L}}
        \mathrm{Ext}^{1}_{\mathcal C_{\widetilde{Q}}}(E,L) \xrightarrow{g_{L}} \mathrm{Ext}^{1}_{\mathcal C_{\widetilde{Q}}}(M,L)
         \longrightarrow coker g_{L}\longrightarrow 0$$
Thus we have
        $$\dim_{k}ker f_{L} + \dim_{k} \mathrm{Ext}^{1}_{\mathcal C_{\widetilde{Q}}}(E,L)+ \dim_{k} coker g_{L} =\dim_{k}
        \mathrm{Ext}^{1}_{\mathcal C_{\widetilde{Q}}}(N,L)+\dim_{k} \mathrm{Ext}^{1}_{\mathcal C_{\widetilde{Q}}}(M,L)$$
Hence
        $$\mathrm{dim}_{k}\mathrm{Ext}^{1}_{\mathcal C_{\widetilde{Q}}}(E,N) \leq
        \mathrm{dim}_{k}\mathrm{Ext}^{1}_{\mathcal C_{\widetilde{Q}}}(N,N)+\mathrm{dim}_{k}\mathrm{Ext}^{1}_{\mathcal C_{\widetilde{Q}}}(M,N)$$
$$\dim_{k}\mathrm{Ext}^{1}_{\mathcal C_{\widetilde{Q}}}(E,E) \leq \dim_{k}\mathrm{Ext}^{1}_{\mathcal
C_{\widetilde{Q}}}(N,E)+\mathrm{dim}_{k}\mathrm{Ext}^{1}_{\mathcal
C_{\widetilde{Q}}}(M,E).$$ Note that $0\neq\epsilon\in ker f_{M}$,
so we have
         $$\dim_{k}\mathrm{Ext}^{1}_{\mathcal C_{\widetilde{Q}}}(E,M) < \dim_{k}\mathrm{Ext}^{1}_{\mathcal C_{\widetilde{Q}}}(N,M)+\dim_{k}\mathrm{Ext}^{1}_{\mathcal C_{\widetilde{Q}}}(M,M).$$
Therefore
        \begin{align*}
\dim_{k}\mathrm{Ext}^{1}_{\mathcal C_{\widetilde{Q}}}(M \oplus N, M
\oplus N)
                &> \dim_{k}\mathrm{Ext}^{1}_{\mathcal C_{\widetilde{Q}}}(E,N) + \dim_{k}\mathrm{Ext}^{1}_{\mathcal C_{\widetilde{Q}}}(E,M) \\
                &=\dim_{k}\mathrm{Ext}^{1}_{\mathcal
C_{\widetilde{Q}}}(N,E)+\dim_{k}\mathrm{Ext}^{1}_{\mathcal C_{\widetilde{Q}}}(M,E)\\
                & \geq \dim_{k}\mathrm{Ext}^{1}_{\mathcal C_{\widetilde{Q}}}(E,E).
\end{align*}This proves our assertion.\end{proof}
\textit{Proof of  Theorem \ref{main}:} We  need to prove that for
any cluster character $X_{L}\in\mathcal{AH}_{|k|}(Q)$, then
$X_{L}\in\mathcal{EH}_{|k|}(Q)$.

Let $L\cong \bigoplus_{i=1}^{l}L_{i}^{\oplus n_{i}}, n_{i}\in
\mathbb{N}$ where $L_{i}\ (1\leq i\leq l)$ are indecomposable
objects in $\mathcal C_{\widetilde{Q}}$. Thus $X_{L_{i}}\ (1\leq
i\leq l)$ are in $\mathcal{EH}_{|k|}(Q)$. By Theorem \ref{hall
multi}, Proposition \ref{exchange2}  and Lemma \ref{easy}, we have
that
$$X^{n_{1}}_{L_{1}}X^{n_{2}}_{L_{2}}\cdots X^{n_{l}}_{L_{l}}=q^{\frac{1}{2}n_{L}}X_{L}+
\sum_{\dim_{k}\mathrm{Ext}^{1}_{\mathcal
C_{\widetilde{Q}}}(E,E)<\dim_{k}\mathrm{Ext}^{1}_{\mathcal
C_{\widetilde{Q}}}(L,L)}f_{n_{E}}(q^{\pm\frac{1}{2}})X_E$$ where
$n_{L}\in \mathbb{Z}$ and $f_{n_{E}}(q^{\pm\frac{1}{2}})\in
\mathbb{Z}[q^{\pm\frac{1}{2}}].$ Thus by induction, we can prove
that $X_{L}\in \mathcal{EH}_{|k|}(Q)$ which implies
$\mathcal{EH}_{|k|}(Q)=\mathcal{AH}_{|k|}(Q).$
\noindent\\[4mm]

\noindent\bf{\footnotesize Acknowledgements}\quad\rm {\footnotesize
The author would like
 to  thank Professor Jie Xiao, Doctor Fan Xu and Doctor Jie Sheng for very helpful conversations.}\\[4mm]

\noindent{\bbb{References}}

\end{document}